\documentclass[letter,11pt]{amsart}
\usepackage{amsmath, amsfonts, amssymb,amscd,graphics,graphicx,color,eucal,setspace} 
\usepackage{latexsym}\def\qed{\hfill$\Box$}   
\usepackage{color}

\numberwithin{equation}{section}

\theoremstyle{plain}
\newtheorem{thm}{Theorem}[section]
\newtheorem{lem}[thm]{Lemma} 

\newtheorem{prop} [thm] {Proposition} 

\theoremstyle{definition}
\newtheorem*{defn}{Definition}
\newtheorem*{rmk}{Remark}
\newtheorem*{ex}{Example}

\newcommand{\lamsum}[1]{\lambda_1+\cdots+\lambda_{#1}}
\newcommand{\Spr}{\mathcal{S}_{\lambda }}

\newcommand{\yy}{y}
\newcommand{\uu}{u}
\newcommand{\apfir}{r}
\newcommand{\dpara}{d}
\newcommand{\wN}{\bar w}
\newcommand{\checklambda}{\check{\lambda}}

\begin{document}
  
\title[The torus equivariant cohomology rings of Springer varieties]{The torus equivariant cohomology rings of \\ Springer varieties}
\author {Hiraku Abe} 
\address{Osaka City University Advanced Mathematical Institute (OCAMI), Sumiyoshi-ku, Osaka 558-8585, Japan}
\email{hirakuabe@globe.ocn.ne.jp}
\author {Tatsuya Horiguchi}
\address{Department of Mathematics, Osaka City University, Sumiyoshi-ku, Osaka 558-8585, Japan}
\email{d13saR0z06@ex.media.osaka-cu.ac.jp}
\date{\today}
\maketitle

\begin{abstract}
The Springer variety of type $A$ associated to a nilpotent operator on $\mathbb{C}^n$ in Jordan canonical form admits a natural action of the $\ell$-dimensional torus $T^{\ell}$ where $\ell$ is the number of the Jordan blocks. 
We give a presentation of the $T^{\ell}$-equivariant cohomology ring of the Springer variety through an explicit construction of an action of the $n$-th symmetric group on the $T^{\ell}$-equivariant cohomology group.
The $T^{\ell}$-equivariant analogue of so called Tanisaki's ideal will appear in the presentation.
\end{abstract}

\setcounter{tocdepth}{1}
\tableofcontents

\section{Introduction} \label{sect:1}
The Springer variety of type $A$ associated to a nilpotent operator $N:\mathbb{C}^n\rightarrow \mathbb{C}^n$ is a closed subvariety of the flag variety of $\mathbb{C}^n$ defined by
\[
\text{$\{V_{\bullet} \in Flags(\mathbb{C}^n)\mid NV_i\subseteq V_{i-1} \ $for all$ \ 1\leq i\leq n \}$}.
\]
When the operator $N$ is in Jordan canonical form with Jordan blocks of weakly decreasing size $\lambda=(\lambda_1,\cdots,\lambda_{\ell})$, we denote the Springer variety by $\mathcal{S}_{\lambda}$.
In 1970's, Springer constructed a representation of the $n$-th symmetric group $S_n$ on the cohomology group $H^*(\mathcal{S}_{\lambda};\mathbb{C})$, and this representation on the top degree part is the irreducible representation of type $\lambda$ (\cite{spr1}, \cite{spr2}).
Tanisaki \cite{t} used this representation to give a simple presentation of the cohomology ring $H^*(\mathcal{S}_{\lambda};\mathbb{C})$ as the quotient of the polynomial ring by an ideal, called Tanisaki's ideal. 
We remark that his argument in \cite{t} works also over $\mathbb{Z}$-coefficient.
%
Our goal in this paper is to give an explicit presentation of the $T^{\ell}$-equivariant cohomology ring $H^*_{T^{\ell}}(\mathcal{S}_{\lambda};\mathbb{Z})$ where we will explain the $\ell$-dimensional torus $T^{\ell}$ below.
In more detail, we will give a presentation as the quotient of the polynomial ring by an ideal whose generators are generalization of the generators of Tanisaki's ideal given in \cite{t}.
Through the the forgetful map $H^*_{T^{\ell}}(\mathcal{S}_{\lambda};\mathbb{Z})\rightarrow H^*(\mathcal{S}_{\lambda};\mathbb{Z})$, our presentation naturally induces the presentation of $H^*(\mathcal{S}_{\lambda};\mathbb{Z})$ given in \cite{t}.

We organize this paper as follows. In Section \ref{sect:2}, we introduce a natural action of the $\ell$-dimensional torus $T^{\ell}$ on the Springer variety $\mathcal{S}_{\lambda}$ for $\lambda=(\lambda_1,\cdots,\lambda_{\ell})$ and give the $T^{\ell}$-fixed points $\mathcal{S}_{\lambda}^{T^{\ell}}$ of the Springer variety $\mathcal{S}_{\lambda}$ where $T^{\ell}$ is defined by the following diagonal unitary matrices:
\begin{equation*} 
 \left\{ \begin{pmatrix}
  h_1E_{\lambda _1}  &    &    &     \\ 
    &  h_2E_{\lambda _2}  &    &         \\
    &    &  \ddots  &         \\
    &    &    &      h_{\ell}E_{\lambda _{\ell}}  
\end{pmatrix}
\mid  \; h_i\in\mathbb{C}, |h_i|=1 \ (1\leq i\leq \ell) \right\}.
\end{equation*}
Here, $E_i$ is the identity matrix of size $i$. 
We construct an $S_n$-action on the equivariant cohomology group $H_{T^{\ell}}^*(\mathcal{S}_{\lambda};\mathbb{Z})$ in Section \ref{sect:3} by using the localization technique which involves the equivariant cohomology of the $T^{\ell}$-fixed points. 
We state the main theorem in Section \ref{sect:4}, and prove it in Section \ref{sect:5} by using this $S_n$-action on $H_{T^{\ell}}^*(\mathcal{S}_{\lambda};\mathbb{Z})$.
Our method of the proof is the $T^{\ell}$-equivariant analogue of \cite{t}.

\medskip
\noindent
\textbf{Acknowledgements.}
The authors thank Professor Toshiyuki Tanisaki for valuable suggestions and kind teachings.

\section{Nilpotent Springer varieties and $T^\ell$-fixed points} \label{sect:2}

We begin with a definition of type $A$ nilpotent Springer varieties. We work with type $A$ in this paper and hence omit it in the following. We first recall that a flag variety $Flags(\mathbb{C}^n)$ consists of nested subspaces of $\mathbb{C}^n$: 
$$V_{\bullet}=(0 = V_0 \subset  V_1 \subset  \dots \subset  V_{n-1} \subset  V_n=\mathbb{C}^n)$$
where $\dim_{\mathbb{C}}V_i=i$ for all $i$.

\begin{defn} 
Let $N\colon \mathbb{C}^n\to \mathbb{C}^n$ be a nilpotent operator. The \textbf{(nilpotent) Springer variety} $\mathcal{S}_N$ associated to $N$ is the set of flags $V_{\bullet}$ satisfying 
 $NV_i\subseteq V_{i-1} \ $for all$ \ 1\leq i\leq n $.
\end{defn}

Since $\mathcal{S}_{gNg^{-1}}$ is homeomorphic (in fact, isomorphic as algebraic varieties) to $\mathcal{S}_N$ for any invertible matrix $g\in GL_n(\mathbb{C})$, we may assume that $N$ is a Jordan canonical form. 
In this paper, we consider the Springer variety 
\[
\text{$\Spr:=\{V_{\bullet} \in Flags(\mathbb{C}^n)\mid N_0V_i\subseteq V_{i-1} \ $for all$ \ 1\leq i\leq n \}$}
\]
where $N_0$ is in Jordan canonical form with Jordan blocks of weakly decreasing sizes $\lambda =(\lambda _1, \lambda _2, \dots, \lambda _{\ell})$.

Let $T^n$ be a $n$-dimensional torus consisting of diagonal unitary matrices:

\begin{equation} \label{eq:T^n}
T^n=\left\{ \begin{pmatrix}
  g_1  &    &    &     \\ 
    &  g_2  &    &         \\
    &    &  \ddots  &         \\
    &    &    &      g_n  
\end{pmatrix} \mid  \; g_i\in\mathbb{C}, |g_i|=1 \ (1\leq i\leq n) \right\}.
\end{equation}
Then the $n$-dimensional torus $T^n$ naturally acts on the flag variety $Flags(\mathbb{C}^n)$, but $T^n$ does not preserve the Springer variety $\Spr$ in general. So we introduce the following $\ell$-dimensional torus:
\begin{equation} \label{eq:T^l}
T^\ell=\left\{ \begin{pmatrix}
  h_1E_{\lambda _1}  &    &    &     \\ 
    &  h_2E_{\lambda _2}  &    &         \\
    &    &  \ddots  &         \\
    &    &    &      h_{\ell}E_{\lambda _{\ell}}  
\end{pmatrix}\in T^n \mid  \; h_i\in\mathbb{C}, |h_i|=1 \ (1\leq i\leq \ell) \right\}
\end{equation}
where $E_i$ is the identity matrix of size $i$ and $\lambda =(\lambda _1, \lambda _2, \dots, \lambda _{\ell})$. 
Then the torus $T^{\ell}$ preserves the Springer variety $\mathcal{S}_{\lambda}$.
Our goal in this section is to give the $T^\ell$-fixed point set $\Spr^{T^\ell}$.

The $T^n$-fixed point set $Flags(\mathbb{C}^n)^{T^n}$ of the flag variety $Flags(\mathbb{C}^n)$ is given by 
$$\{(\langle e_{w(1)}\rangle \subset \langle e_{w(1)},e_{w(2)}\rangle \subset \dots \subset \langle e_{w(1)},e_{w(2)},\dots ,e_{w(n)}\rangle =\mathbb{C}^n) \mid w\in S_n\}$$
where $e_1,e_2,\dots ,e_n$ is the standard basis of $\mathbb{C}^n$ and $S_n$ is the symmetric group on $n$ letters $\{1,2,\dots ,n \}$, so we may identify $Flags(\mathbb{C}^n)^{T^n}$ with $S_n$.

Let $w$ be an element of $S_n$ satisfying the following property:
\begin{align} \label{eq:2.2}
&\text{for each $1\leq k\leq\ell$, the numbers between $\lamsum{k-1}+1$ and $\lamsum{k}$} \\ \notag
&\text{appear in the one-line notation of $w$ as a subsequence in the increasing order.}
\end{align}
Here, we write $\lamsum{k-1}+1=1$ when $k=1$.



%
%

%
%
\begin{ex}
We consider the case $n=6$, $\ell=3$, and $\lambda =(3,2,1)$. Using one-line notation, the following permutations 
$$w_1=124365, \ \ w_2=416253, \ \ w_3=612435$$
satisfy the condition \eqref{eq:2.2}. In fact, the sequences $(1,2,3)$, $(4,5)$ and $(6)$ appear in the one-line notations as a subsequence in the increasing order.  
\end{ex}

\begin{lem}\label{lem2.1}
The $T^\ell$-fixed points $\Spr^{T^\ell}$ of the Springer variety $\Spr$ is the set 
\begin{equation*}
\{w\in S_n \mid w \textrm{ satisfy the condition } \emph(2.3) \}.
\end{equation*}
\end{lem}

\begin{proof}
Let $w=V_\bullet$ a permutation satisfying the condition \eqref{eq:2.2}.
Since $w(1)$ is equal to one of the numbers $1$, $\lambda_1+1$, $\lambda_1+\lambda_2+1$, \dots, $\lamsum{\ell-1}+1$, we have $N_0V_1\subseteq \{0\}$. 
If $w(1)=\lamsum{k-1}+1$, then $w(2)$ is equal to one of the numbers $1$ ,$\lambda _1+1$, $\dots,\lamsum{k-1}+2$, \dots, $\lamsum{\ell-1}+1$. So we also have $N_0V_2\subseteq V_1$. Continuing  this argument, we have $N_0V_i\subseteq V_{i-1}$ for all $1\leq i\leq n$, and it follows that the $w$ is an element of $\Spr$. On the other hand, the $w$ is clearly fixed by $T^\ell$, so the $w$ is an element of $\Spr^{T^\ell}$.

Conversely, let $V_\bullet$ be an element of $\Spr^{T^\ell}$. Let $v_1$,$v_2$,$\dots$,$v_j$ be generators for $V_j$ where $v_j=(x^{(j)}_{1},x^{(j)}_{2},\cdots,x^{(j)}_{n})^t$ in $\mathbb{C}^n$ for all $j$.
Since we have 
\begin{align*} 
&N_0v_1=(\underbrace {x^{(1)}_{2},\cdots,x^{(1)}_{\lambda _1},0}_{\lambda_1},
       \underbrace {x^{(1)}_{\lambda _1+2 },\cdots,x^{(1)}_{\lambda _1+\lambda _2},0}_{\lambda_2},
       \cdots\cdots,
       \underbrace {x^{(1)}_{\lamsum{\ell-1}+2 },\cdots,x^{(1)}_{n},0}_{\lambda_{\ell}})^t, 
\end{align*} 
the condition $N_0V_1\subseteq V_0=\{0 \}$ implies that
\begin{equation} \label{eq:2.3}
v_1=(\underbrace {x^{(1)}_{1},0,\cdots,0}_{\lambda_1},
     \underbrace {x^{(1)}_{\lambda _1+1 },0,\cdots,0}_{\lambda_2},
     \cdots\cdots,
     \underbrace {x^{(1)}_{\lamsum{\ell-1}+1 },0,\cdots,0}_{\lambda_{\ell}})^t.
\end{equation}
It follows that exactly one of $x^{(1)}_{i}$\ ($i=1,\lambda _1+1,\lambda _1+\lambda _2+1,\dots,\lamsum{\ell-1}+1$) which appear in \eqref{eq:2.3} is nonzero. In fact, $V_\bullet$ is fixed by the $T^{\ell}$-action and hence we have $h\cdot v_1=v_1$ for arbitrary $h\in T^\ell$ where
\begin{equation*} 
h\cdot v_1=(\underbrace {h_1x^{(1)}_{1},0,\cdots,0}_{\lambda_1},
            \underbrace {h_2x^{(1)}_{\lambda _1+1 },0,\cdots,0}_{\lambda_2},
            \cdots\cdots,
            \underbrace {h_{\ell}x^{(1)}_{\lamsum{\ell-1}+1 },0,\cdots,0}_{\lambda_{\ell}})^t.
\end{equation*}
Since each $h_i$ runs over all complex numbers whose absolute values are 1, only one of $x^{(1)}_{i}$ in \eqref{eq:2.3} must be nonzero.

If $x^{(1)}_{\lamsum{k-1}+1 }$ is nonzero for some $1\leq k\leq \ell$, then we may assume that
\begin{align*} 
&v_1=(0,\cdots,0,1,0,\cdots,0)^t, \\
&v_j=(x^{(j)}_{1},\cdots,x^{(j)}_{\lamsum{k-1} },0,x^{(j)}_{\lamsum{k-1}+2 },\cdots,x^{(j)}_{n})^t
\end{align*} 
for $2\leq j\leq n$ where the ($\lamsum{k-1}+1$)-th component of $v_1$ is one.
Since we have 
\begin{align*} 
&N_0v_2=(\underbrace {x^{(2)}_{2},\cdots,x^{(2)}_{\lambda _1},0}_{\lambda_1},
       \underbrace {x^{(2)}_{\lambda _1+2 },\cdots,x^{(2)}_{\lambda _1+\lambda _2},0}_{\lambda_2},
       \cdots\cdots,
       \underbrace {x^{(2)}_{\lamsum{\ell-1}+2 },\cdots,x^{(2)}_{n},0}_{\lambda_{\ell}})^t, 
\end{align*} 
the condition $N_0V_2\subseteq V_1$ implies that
\begin{equation} \label{eq:2.4}
v_2=(\underbrace {x^{(2)}_{1},0,\cdots,0}_{\lambda_1},
     \cdots\cdots,
     \underbrace {0,x^{(2)}_{\lamsum{k-1}+2 },0,\cdots,0}_{\lambda_k},
     \cdots\cdots,
     \underbrace {x^{(2)}_{\lamsum{\ell-1}+1 },0,\cdots,0}_{\lambda_{\ell}})^t.
\end{equation}
Therefore, we see that the only one of $x^{(2)}_{i}$\ ($i=1,\lambda _1+1,\dots,\lamsum{k-1}+2,\dots,\lamsum{\ell-1}+1$) which appear in \eqref{eq:2.4} is nonzero by an argument similar to that used above. Continuing this procedure, we conclude that 
$V_\bullet =w$
for some $w\in S_n$ satisfying the condition \eqref{eq:2.2}.
In fact, $w(1)$ is equal to one of the numbers $1$, $\lambda_1+1$, $\lambda_1+\lambda_2+1$, \dots, $\lamsum{\ell-1}+1$. If $w(1)=\lamsum{k-1}+1$, then $w(2)$ is equal to one of the numbers $1$ ,$\lambda _1+1$, $\dots,\lamsum{k-1}+2$, \dots, $\lamsum{\ell-1}+1$ and so on. This means that for each $k=1,\dots,\ell$ the numbers between $\lamsum{k-1}+1$ and $\lamsum{k}$ appear in the one-line notation of $w$ as a subsequence in the increasing order.
\end{proof}

Regarding a product of symmetric groups $S_{\lambda _1}\times S_{\lambda _2}\times \cdots \times S_{\lambda _{\ell}}$ as a subgroup of the symmetric group $S_n$, it follows from Lemma~\ref{lem2.1} that the $T^\ell$-fixed points $\Spr^{T^\ell}$ of the Springer variety $\Spr$ is identified with the right cosets $S_{\lambda _1}\times S_{\lambda _2}\times \cdots \times S_{\lambda _{\ell}}\backslash S_n$ where  each $w\in\Spr^{T^\ell}$ corresponds to the right coset $[w]$. 
In fact, the condition \eqref{eq:2.2} provides a unique representative for each right coset.

\section{An action of the symmetric group $S_n$ on $H^{\ast}_{T^\ell}(\Spr)$} \label{sect:3}

In this section, we introduce an action of the symmetric group $S_n$ on the equivariant cohomology group $H^{\ast}_{T^\ell}(\Spr)$ over $\mathbb{Z}$-coefficient by using the localization technique.
We will see that the projection map
$$\rho_{\lambda}: H^{\ast}_{T^n}(Flags(\mathbb{C}^n))\rightarrow H^{\ast}_{T^{\ell}}(\Spr)$$
induced from the inclusions of $\Spr$ into $Flags(\mathbb{C}^n)$ and $T^{\ell}$ into $T^n$ is $S_n$-equivariant map.
In particular, we consider the following commutative diagram
\begin{equation}\label{eq:3.3}
\begin{CD}
H^{\ast}_{T^n}(Flags(\mathbb{C}^n))@>{\iota _1}>> H^{\ast}_{T^n}(Flags(\mathbb{C}^n)^{T^n})=\displaystyle \bigoplus_{w\in S_n} \mathbb{Z}[t_1,\dots,t_n]\\
@V{\rho_{\lambda}}VV @V{\pi }VV\\
H^{\ast}_{T^{\ell}}(\Spr)@>{\iota _2}>> H^{\ast}_{T^{\ell}}(\Spr^{T^{\ell}})=\displaystyle \bigoplus_{w\in \Spr^{T^\ell}\subseteq S_n} \mathbb{Z}[u_1,\dots,u_{\ell}]
\end{CD}
\end{equation}
where all the maps are induced from inclusion maps, and construct $S_n$-actions on the three modules $H^{\ast}_{T^n}(Flags(\mathbb{C}^n))$, $\bigoplus_{w\in S_n} \mathbb{Z}[t_1,\dots,t_n]$, and $\bigoplus_{w\in \Spr^{T^\ell}} \mathbb{Z}[u_1,\dots,u_{\ell}]$ to construct an $S_n$-action on $H^{\ast}_{T^{\ell}}(\Spr)$.
All (equivariant) cohomology rings are assumed to be over $\mathbb{Z}$-coefficient unless otherwise specified. 

First, we introduce the left action of the symmetric group $S_n$ on the cohomology group $H^*(Flags(\mathbb{C}^n))$. To do that, we consider the right $S_n$-action on the flag variety $Flags(\mathbb{C}^n)$ as follows.

For any $V_\bullet \in Flags(\mathbb{C}^n)$, there exists $g\in U(n)$ so that $V_i=\bigoplus_{j=1}^{i}\mathbb{C}g(e_j)$, where $\{e_1,\dots,e_n \}$ is the standard basis of $\mathbb{C}^n$. Then the right action of $w\in S_n$ on $Flags(\mathbb{C}^n)$ can be defined by
\begin{equation} \label{eq:3.1}
V_\bullet\cdot w=V'_\bullet
\end{equation}
where $V'_i=\bigoplus_{j=1}^{i}\mathbb{C}g(e_{w(j)})$.

We recall an explicit presentation of the $T^n$-equivariant cohomology ring of the flag variety $Flags(\mathbb{C}^n)$. 
Let $E_i$ be the subbundle of the trivial vector bundle $Flags(\mathbb{C}^n)\times \mathbb{C}^n$ over $Flags(\mathbb{C}^n)$ whose fiber at a flag $V_{\bullet}$ is just $V_i$. We denote the $T^n$-equivariant first Chern class of the line bundle $E_i/E_{i-1}$ by $\bar x_i\in H^2_{T^n}(Flags(\mathbb{C}^n))$. 
Let $\mathbb{C}_i$ be the one dimensional representation of $T^n$ through a map $T^n\rightarrow S^1$ given by $diag(g_1,\dots,g_n)\mapsto g_i$. We denote the first Chern class of the line bundle $ET^n\times _{T^n}\mathbb{C}_i$ over $BT^n$ by $t_i\in H^2(BT^n)$. Since $t_1,\dots,t_n$ generate $H^*(BT^n)$ as a ring and they are algebraically independent, we identify $H^*(BT^n)$ with a polynomial ring $\mathbb{Z}[t_1,\dots,t_n]$.
Then $H^{\ast}_{T^n}(Flags(\mathbb{C}^n))$ is generated by $\bar x_1,\dots ,\bar x_n,t_1,\dots ,t_n$ as a ring. Defining a surjective ring homomorphism from the polynomial ring $\mathbb{Z}[x_1,\dots ,x_n,t_1,\dots,t_n]$ to $H^{\ast}_{T^n}(Flags(\mathbb{C}^n))$ by sending $x_i$ to $\bar x_i$ and $t_i$ to $t_i$, its kernel $\tilde I$ is generated as an ideal by $e_i(x_1,\dots ,x_n)-e_i(t_1,\dots ,t_n)$ for all $1\leq i\leq n$, where $e_i$ is the $i$-th elementary symmetric polynomial. Thus, we have an isomorphism
\begin{equation}\label{eq:3.1.5}
H^{\ast}_{T^n}(Flags(\mathbb{C}^n))\cong \mathbb{Z}[x_1,\dots,x_n,t_1,\dots,t_n]/\tilde I.
\end{equation}

The right action in \eqref{eq:3.1} induces the following left action of the symmetric group $S_n$ on $H^{\ast}_{T^n}(Flags(\mathbb{C}^n))$:
\begin{equation}\label{eq:3.2}
w\cdot \bar x_i=\bar x_{w(i)}, \ w\cdot t_i=t_i
\end{equation}
for $w\in S_n$.
In fact, the pullback of the line bundle $E_i/E_{i-1}$ under the right action in \eqref{eq:3.1} is exactly the line bundle $E_{w(i)}/E_{w(i)-1}$, and the right action in \eqref{eq:3.1} is $T^n$-equivariant.

Second, we define a left action of $v\in S_n$ on the direct sum $\bigoplus_{w\in S_n} \mathbb{Z}[t_1,\dots,t_n]$ of the polynomial ring as follows:
\begin{equation}\label{eq:3.8}
(v\cdot f)|_w=f|_{wv}
\end{equation}
where $w\in S_n$ and $f\in \bigoplus_{w\in S_n} \mathbb{Z}[t_1,\dots,t_n]$. 
Observe that the map $\iota _1$ in \eqref{eq:3.3} is the following mapping
\begin{align} 
&\iota _1(\bar x_i)|_w=t_{w(i)}, \ \iota _1(t_i)|_w=t_{i}.\label{eq:3.5} 
\end{align} 
Note that it follows from \eqref{eq:3.2}, \eqref{eq:3.8}, and \eqref{eq:3.5} that the map $\iota _1$ is $S_n$-equivariant map, i.e. $w\cdot (\iota _1(f))=\iota _1(w\cdot f)$ for any $f\in H^{\ast}_{T^n}(Flags(\mathbb{C}^n))$ and $w\in S_n$. 

To construct an $S_n$-action on $\bigoplus_{w\in \Spr^{T^\ell}} \mathbb{Z}[u_1,\dots,u_{\ell}]$ 
, we need some preparations. 
 We identify $H^*(BT^{\ell})$ with a polynomial ring with $\ell$ variables. That is,  
$$H^{\ast}(BT^{\ell})=\mathbb{Z}[u_1,\dots,u_{\ell}]$$
where $u_i\in H^2(BT^{\ell})$ is the first Chern class of the line bundle $ET^{\ell}\times _{T^{\ell}}\mathbb{C}_i$ over $BT^{\ell}$. Here, $\mathbb{C}_i$ is the one dimensional representation of $T^{\ell}$ through a map $T^{\ell}\rightarrow S^1$ given by 
diag$(h_1,\cdots,h_1,h_2,\cdots,h_2,\cdots\cdots,h_{\ell},\cdots,h_{\ell}) \mapsto h_i$.

It is known that $Flags(\mathbb{C}^n)$ and $\Spr$ admit a cellular decomposition (\cite{spa}), so the odd degree cohomology groups of $Flags(\mathbb{C}^n)$ and $\Spr$ vanish. The path-connectedness of $Flags(\mathbb{C}^n)$ and $\Spr$ together with this fact implies that the maps $\iota _1$ and $\iota _2$ in \eqref{eq:3.3} are injective (cf.\cite{m-t}) and that the map $\rho_{\lambda}$ in \eqref{eq:3.3} is surjective (cf.\cite{h-s}). The map $\pi $ in \eqref{eq:3.3} is clearly surjective.
Therefore, we obtain the following lemma. Let $\bar y_i$ be the image $\rho_{\lambda}(\bar x_i)$ of $\bar x_i$ for each $i$.


\begin{lem}\label{lem3.1} 
The $T^{\ell}$-equivariant cohomology ring $H^{\ast}_{T^{\ell}}(\Spr)$ is generated by $\bar y_1$,$\dots$,$\bar y_n$,\\
$u_1$,$\dots$,$u_{\ell}$ as a ring where $\bar y_i$ is as above and $H^{\ast}(BT^{\ell})=\mathbb{Z}[u_1,\dots,u_{\ell}]$. \qed
\end{lem}

Let $\phi $ : $[n]\rightarrow [\ell]$ ($[n]:=\{1,2,\dots,n \}$) be a map defined by 
\begin{equation}\label{eq:3.4}
\phi(i)=k
\end{equation}
if $\lamsum{k-1}+1\leq i\leq \lamsum{k}$ where $\lamsum{k-1}=0$ when $k=1$.
Observe that the map $\pi $ in \eqref{eq:3.3} is the following mapping
\begin{align} 
&\pi (f|_{w}(t_1,\dots,t_n))=f|_{w}(u_{\phi(1)},\dots,u_{\phi(n)}),\label{eq:3.6}
\end{align} 
where $f|_{w}$ denotes $w$-component of $f$.
It follows from \eqref{eq:3.5}, \eqref{eq:3.6} and the commutative diagram in \eqref{eq:3.3} that
\begin{equation}\label{eq:3.7}
\iota _2(\bar y_i)|_{w}=u_{\phi (w(i))} \ \textrm{and} \ \iota _2(u_i)|_{w}=u_{i}.
\end{equation}

Third, we define the left action of $v\in S_n$ on the direct sum $\bigoplus_{w\in \Spr^{T^\ell}} \mathbb{Z}[u_1,\dots,u_{\ell}]$ of the polynomial ring as follows:
\begin{equation}\label{eq:3.9}
(v\cdot f)|_{w}=f|_{w'}
\end{equation}
 for $w\in \Spr^{T^\ell}$ and $f\in \bigoplus_{w\in \Spr^{T^\ell}} \mathbb{Z}[u_1,\dots,u_{\ell}]$ where $w'$ is the element of $\Spr^{T^\ell}$ whose right coset agrees with the right coset $[wv]$ of $S_{\lambda _1}\times S_{\lambda _2}\times \cdots \times S_{\lambda _{\ell}}\backslash S_n$. 
Note that the map $\pi $ in \eqref{eq:3.3} is not $S_n$-equivariant in general. 

\begin{lem}\label{lem3.2} 
For any $v\in S_n$ and $1\leq i\leq n$, it follows that
\begin{equation} \label{eq:3.10}
v\cdot (\iota _2(\bar y_i))=\iota _2(\bar y_{v(i)}) \ \textrm{and} \ v\cdot (\iota _2(u_i))=\iota _2(u_i) 
\end{equation} 
where the map $\iota _2$ is in \eqref{eq:3.3} and $\bar y_i$ is the image of $\bar x_i$ under the map $\rho_{\lambda}$ in \eqref{eq:3.3}. 
\end{lem}

\begin{proof}
From \eqref{eq:3.7} and \eqref{eq:3.9}, we have 
$$(v\cdot (\iota _2(u_i)))|_w=\iota _2(u_i)|_{w'}=u_i=\iota _2(u_i)|_w$$
for all $w\in S_n$. So the second equation holds.
From \eqref{eq:3.7} and \eqref{eq:3.9} again, we have 
\begin{align*} 
&(v\cdot (\iota _2(\bar y_i)))|_w=\iota _2(\bar y_i)|_{w'}=u_{\phi (w'(i))}, \\ 
&\iota _2(\bar y_{v(i)})|_w=u_{\phi (w(v(i)))}.
\end{align*} 
Therefore, it is enough to prove $\phi (w'(i))=\phi (wv(i))$.
Since $[w']=[wv]$ in $S_{\lambda _1}\times S_{\lambda _2}\times \cdots \times S_{\lambda _{\ell}}\backslash S_n$, we have 
\begin{align*} 
&\lamsum{r-1}+1\leq w'(i)\leq \lamsum{r}, \\
&\lamsum{r-1}+1\leq wv(i)\leq \lamsum{r}
\end{align*} 
for some $r$. 
From the definition \eqref{eq:3.4} of the map $\phi $, we have $\phi (w'(i))=\phi (wv(i))$, and the first equation holds. We are done.
\end{proof}

Since the map $\iota _2$ is injective, we obtain an $S_n$-action on $H^{\ast}_{T^{\ell}}(\Spr)$ satisfying 
\begin{equation} \label{eq:3.11}
w\cdot \bar y_i=\bar y_{w(i)} \ \textrm{and} \ w\cdot u_i=u_i
\end{equation} 
for $w\in S_n$ from Lemma~\ref{lem3.1} and Lemma~\ref{lem3.2}.
Moreover, one can see that the map $\rho_{\lambda}$ in \eqref{eq:3.3} is $S_n$-equivariant homomorphism by \eqref{eq:3.2} and \eqref{eq:3.11}. We summarize the results in this section as follows. 

\begin{prop}\label{prop3.3} 
There exists an $S_n$-action on $H^{\ast}_{T^{\ell}}(\Spr)$ such that the map $\rho_{\lambda}$ in \eqref{eq:3.3} is $S_n$-equivariant homomorphism where the $S_n$-action on $H^{\ast}_{T^n}(Flags(\mathbb{C}^n))$ is given by \eqref{eq:3.2}.
\end{prop}

\section{Main theorem} \label{sect:4}
In this section, we state our main theorem. For this purpose, let us clarify our notations.
We set $p_{\lambda}(s):=\lambda_{n-s+1}+\lambda_{n-s+2}+\cdots+\lambda_{\ell}$ for $s=1,\cdots,n$.
We denote by $\checklambda$ the transpose of $\lambda$. That is, $\checklambda=(\eta_1,\cdots,\eta_k)$ where $k=\lambda_1$ and $\eta_i = |\{j \mid \lambda_j\geq i\}|$ for $1\leq i\leq k$. 
For indeterminates $y_1,\cdots,y_s$ and $a_1,a_2,\cdots$, let
\begin{equation}\label{eq:4.1}
e_{d}(y_1,\cdots,y_s | a_1,a_2,\cdots) := \sum_{r=0}^d (-1)^{d-r} e_r(y_1,\cdots,y_s) h_{d-r}(a_1,\cdots,a_{s+1-d})
\end{equation}
for $d\geq 0$ where $e_i$ and $h_i$ denote the $i$-th elementary symmetric polynomial and the $i$-th complete symmetric polynomial, respectively. 
In fact, this is the factorial Schur function corresponding to the Young diagram consisting of the unique column of length $d$ as shown in the next section (see Lemma \ref{fact Schur by elem sym}).
We also define a map $\phi_{\lambda}:[n]\rightarrow[\ell]$ by the condition 
\begin{align}\label{eq:4.4}
&(u_{\phi_{\lambda}(1)},\cdots,u_{\phi_{\lambda}(n)}) \\ \notag
&\quad=
(\underbrace{u_1,\cdots,u_1}_{\lambda_1-\lambda_2}, 
\underbrace{u_1,u_2,\cdots,u_1,u_2}_{2(\lambda_2-\lambda_3)}, 
\cdots\cdots,
\underbrace{\uu_1,\uu_2,\cdots,\uu_{\ell}, \cdots\cdots,\uu_1,\uu_2,\cdots,\uu_{\ell}}_{\ell(\lambda_\ell-\lambda_{\ell+1})})
\end{align}
as ordered sequences where for each $1\leq \apfir \leq \ell$ the $\apfir$-th sector of the right-hand-side consists of $(u_1,u_2,\cdots,u_\apfir)$ repeated $(\lambda_\apfir-\lambda_{\apfir+1})$-times. Here, we denote $\lambda_{\ell+1}=0$.

Let us define a ring homomorphism
\begin{equation}\label{eq:4.5}
\psi:\mathbb{Z}[\yy_1,\cdots,\yy_n, \uu_1,\cdots,\uu_\ell]\rightarrow H_{T^{\ell}}^*(\mathcal{S}_{\lambda})
\end{equation}
by sending $\yy_i$ to $\bar \yy_i$ and $\uu_i$ to $\uu_i$ where $H^*(BT^{\ell})=\mathbb{Z}[u_1,\cdots,u_{\ell}]$. Recall that $\bar \yy_i$ is the equivariant first Chern class of the tautological line bundle $E_i/E_{i-1}$ over $Flags(\mathbb{C}^n)$ (see Section \ref{sect:3}) restricted to $\mathcal{S}_{\lambda}$. This homomorphism $\psi$ is a surjection by Lemma \ref{lem3.1} and the surjectivity of the projection map $\rho_{\lambda}$. 
\begin{thm}\label{main thm}
The map $\psi$ in \eqref{eq:4.5} induces a ring isomorphism 
\begin{equation*}
H_{T^{\ell}}^*(\mathcal{S}_{\lambda})
\cong 
\mathbb{Z}[\yy_1,\cdots,\yy_n, \uu_1,\cdots,\uu_\ell]/\widetilde{I}_{\lambda}
\end{equation*}
where $\widetilde{I}_{\lambda}$ is the ideal of the polynomial ring $\mathbb{Z}[\yy_1,\cdots,\yy_n, \uu_1,\cdots,\uu_\ell]$ generated by the polynomials $e_{d}(\yy_{i_1},\cdots,\yy_{i_s}|\uu_{\phi_{\lambda}(1)},\cdots,\uu_{\phi_{\lambda}(n)})$ defined in \eqref{eq:4.1} with $\phi_{\lambda}$ described in \eqref{eq:4.4} for $1\leq s\leq n$, $1\leq i_1<\cdots<i_s\leq n$, and $d \geq s+1-p_{\checklambda}(s)$.
\end{thm}

\begin{rmk}
The ideal $\widetilde{I}_{\lambda}$ is the $T^{\ell}$-equivariant analogue of so-called Tanisaki's ideal (it is written as $K_{\checklambda}$ in \cite{t}). Each generator of $\widetilde{I}_{\lambda}$ given above specializes to a generator of Tanisaki's ideal given in \cite{t} after the evaluation $u_i=0$ for all $i$.
\end{rmk}

\section{Proof of the main theorem} \label{sect:5}
In this section, we prove Theorem~\ref{main thm}. Our argument is the $T^{\ell}$-equivariant version of \cite{t}.
We first show that $e_{d}(\bar \yy_{i_1},\cdots,\bar \yy_{i_s}|\uu_{\phi_{\lambda}(1)},\cdots,\uu_{\phi_{\lambda}(n)})=0$ in $H^{\ast}_{T^\ell}(\mathcal{S}_{\lambda})$ for $1\leq s\leq n$, $1\leq i_1<\cdots<i_s\leq n$, and $d \geq s+1-p_{\checklambda}(s)$.
By the $S_n$-action on $H^{\ast}_{T^\ell}(\mathcal{S}_{\lambda})$ constructed in Section \ref{sect:3}, we may assume that $i_1=1,\cdots,i_s=s$. 
We prove the claim for the case $s=n$ before treating general cases $s<n$. In this case, we have that $d \geq n+1-p_{\checklambda}(n)=1$.
Observe that in $H_{T^n}^*(Flags(\mathbb{C}^n))$ we have
\begin{align*}
&e_{\dpara}(\bar x_1,\cdots,\bar x_n | t_1,\cdots,t_n) \\
&\qquad=
\sum_{r=0}^d (-1)^{d-r} e_r(\bar x_1,\cdots, \bar x_n) h_{d-r}(t_1,\cdots,t_{n+1-d}) \\
&\qquad=
\sum_{r=0}^d (-1)^{d-r} e_r(t_1,\cdots, t_n) h_{d-r}(t_1,\cdots,t_{n+1-d})
\end{align*}
by the presentation given in \eqref{eq:3.1.5}. 
It is straightforward to check that this is equal to $e_d(t_{n+2-d},\cdots,t_n)$ (which is zero since the number of variables is greater than $d$) by considering the generating functions with a formal variable $z$ for elementary and complete symmetric polynomials  :
\begin{align*}
&\prod_{i=1}^n (1-t_iz) = \sum_{r=0} ^{n} (-1)^re_r(t_1,\cdots, t_n) z^r, \\
&\prod_{i=1}^n \frac{1}{1-t_iz} = \sum_{r\geq0} h_r(t_1,\cdots, t_n) z^r.
\end{align*}
That is, the polynomial $e_{\dpara}(\bar x_1,\cdots,\bar x_n | t_1,\cdots,t_n)$ vanishes in $H^{\ast}_{T^n}(Flags(\mathbb{C}^n))$, and hence we see that $e_{\dpara}(\bar y_1,\cdots,\bar y_n | u_{\phi_{\lambda}(1)},\cdots,u_{\phi_{\lambda}(n)})=0$.

Let us next consider the cases for $s<n$, and prove that for $d \geq s+1-p_{\checklambda}(s)$ we have $e_{d}(\bar \yy_{1},\cdots,\bar \yy_{s}|\uu_{\phi_{\lambda}(1)},\cdots,\uu_{\phi_{\lambda}(n)})=0$ in $H^{\ast}_{T^\ell}(\mathcal{S}_{\lambda})$. Take a $T^n$-invariant complete flag $U_{\bullet}$ by refining the flag $(\cdots\subset N_0^2\mathbb{C}^n \subset N_0\mathbb{C}^n \subset \mathbb{C}^n)$. This is possible since $N_0$ is in Jordan canonical form.
We denote by $\wN$ the element of $S_n$ corresponding to $U_{\bullet}$, i.e. $U_{\bullet}=\wN F_{\bullet}$ where $F_{\bullet}$ is the standard flag defined by $F_i=\langle e_1, \cdots, e_i \rangle$ for all $1\leq i\leq n$. 
For a Young diagram $\mu$ with at most $s$ rows and $n-s$ columns, the Schubert variety corresponding to $\mu$ with respect to the reference flag $U_{\bullet}$ is
\begin{equation*}
X_{\mu}(U_{\bullet}) = \{ V \in Gr_s(\mathbb{C}^n) \mid \dim(V\cap U_{n-s+i-\mu_i})\geq i \text{ for all } 1\leq i\leq s \}
\end{equation*}
where $Gr_s(\mathbb{C}^n)$ denotes the set of $s$ dimensional complex linear subspaces in $\mathbb{C}^n$.
It is known that $X_{\mu}(\tilde{F}_{\bullet}) \cap X_{\nu}(F_{\bullet}) = 
\emptyset$ unless $\mu \subset \nu^{\dagger}$ (cf. \cite{f} \S\ 9.4, Lemma 3).
Here, $\nu^{\dagger}=(n-s-\nu_s,\cdots,n-s-\nu_1)$ and $\tilde{F}_{\bullet}$ is the opposite flag of $F_{\bullet}$ defined by $\tilde{F}_{i}=\langle e_{n+1-i},\cdots,e_n \rangle$.
By multiplying both sides of this equality by $\wN$, we get
\begin{equation}\label{eq:5.1}
X_{\mu}(\wN\tilde{F}_{\bullet}) \cap X_{\nu}(U_{\bullet}) = 
\emptyset \quad \text{unless } \mu \subset \nu^{\dagger}.
\end{equation}
Since the flag $\wN\tilde{F}_{\bullet}$ is $T^n$-invariant, the Schubert variety $X_{\mu}(\wN\tilde{F}_{\bullet})$ is a $T^n$-invariant irreducible subvariety of $Gr_s(\mathbb{C}^n)$.
Let $\tilde{S}_{\mu}:=[X_{\mu}(\wN\tilde{F}_{\bullet})]\in H^{\ast}_{T^n}(Gr_s(\mathbb{C}^n))$ be the associated $T^n$-equivariant cohomology class.

Let $p:Flags(\mathbb{C}^n)\rightarrow Gr_s(\mathbb{C}^n)$ be the projection defined by $p(V_{\bullet})=V_s$.
Then it follows that
\begin{equation*}
p(\mathcal{S}_{\lambda})\subset X_{\mu_0}(U_{\bullet})
\end{equation*}
where $\mu_0=(n-s,\cdots,n-s,0,\cdots,0)$ with $n-s$ repeated $p_{\checklambda}(s)$-times and $0$ repeated $(s-p_{\checklambda}(s))$-times (cf. \cite{t} \S\ 3, Proposition 3). 
Hence, we obtain the following commutative diagram
\begin{equation}\label{eq:5.2}
\begin{CD}
H^{\ast}_{T^n}(Flags(\mathbb{C}^n))@<{p^*}<< H^{\ast}_{T^n}(Gr_s(\mathbb{C}^n))\\
@V{\rho _\lambda}VV @VV{i^* }V\\
H^{\ast}_{T^{\ell}}(\mathcal{S}_{\lambda})@<{k^*}<< H^{\ast}_{T^n}(X_{\mu_0}(U_{\bullet}))
\end{CD}
\end{equation}
where $i^*$ is the map induced by the inclusion and $k$ is the restriction of the projection map $p$.
Let $\mu_{s,\dpara}=(1,\cdots,1,0,\cdots,0)$ with $1$ repeated $\dpara$-times and $0$ repeated $(s-\dpara)$-times. This Young diagram has at most $s$ rows and $n-s$ columns since we are assuming that $s<n$.
Recall that the $T^n$-equivariant Schubert class $\tilde{S}_{\mu}=[X_{\mu}(\wN\tilde{F}_{\bullet})]$ comes from the relative cohomology $H^{\ast}_{T^n}(Gr_s(\mathbb{C}^n), Gr_s(\mathbb{C}^n)\backslash X_{\mu}(\wN\tilde{F}_{\bullet}))$. 
So it follows that $i^*\tilde{S}_{\mu_{s,\dpara}} = 0$ for $\dpara\geq s+1-p_{\checklambda }(s)$ since $\mu_{s,\dpara}\not\subset\mu_0^{\dagger}$ and (\ref{eq:5.1}) show that any cycle in $X_{\mu_0}(U_{\bullet})$ does not intersect with $X_{\mu_{s,\dpara}}(\wN\tilde{F}_{\bullet})$. Thus, we obtain $\rho_{\lambda}(p^*\tilde{S}_{\mu_{s,\dpara}})=0$ by the commutativity of the diagram \eqref{eq:5.2}.

To give a polynomial representative of $\rho_{\lambda}(p^*\tilde{S}_{\mu_{s,\dpara}})$, let us first describe $p^*\tilde{S}_{\mu_{s,\dpara}}$ in terms of $\bar x_1,\cdots,\bar x_n$ and $t_1,\cdots,t_n$.
Observe that $w\in S_n$ acts on $\mathbb{C}^n$ from the left by 
\begin{equation*}
w\cdot(x_1,\cdots,x_n)=(x_{w^{-1}(1)},\cdots,x_{w^{-1}(n)})
\end{equation*}
for $(x_1,\cdots,x_n)\in\mathbb{C}^n$, and this naturally induces $S_n$-action on $Flags(\mathbb{C}^n)$. 
For each $w\in S_n$, the induced map on $Flags(\mathbb{C}^n)$ is equivariant with respect to a group homomorphism $\psi_w: T^n\rightarrow T^n$ defined by $(g_1,\cdots,g_n)\mapsto (g_{w^{-1}(1)},\cdots,g_{w^{-1}(n)})$. 
This $\psi_w$ induces a ring homomorphism on $H^{\ast}(BT^n)=\mathbb{Z}[t_1,\cdots,t_n]$ :
\begin{equation*}
\psi_w^*: \mathbb{Z}[t_1,\cdots,t_n] \rightarrow \mathbb{Z}[t_1,\cdots,t_n] \quad ; \quad 
t_i\mapsto t_{w^{-1}(i)},
\end{equation*}
and the induced map $w^*$ on $H^{\ast}_{T^n}(Flags(\mathbb{C}^n))$ is a ring homomorphism satisfying $w^*(t_i\alpha)=\psi_w^*(t_i)w^*(\alpha)$ for any $\alpha\in H^{\ast}_{T^n}(Flags(\mathbb{C}^n))$ and $i=1,\cdots,n$ where the products are taken by the cup products via the canonical homomorphism $H^{\ast}(BT^n)\rightarrow H^{\ast}_{T^n}(Flags(\mathbb{C}^n))$.
Similarly, $S_n$ acts on $Gr_s(\mathbb{C}^n)$ from the left, and the projection $p:Flags(\mathbb{C}^n)\rightarrow Gr_s(\mathbb{C}^n)$ is $S_n$-equivariant.
Observe that $w^*\bar x_i=\bar x_i$ for any $w\in S_n$ since the map $w:Flags(\mathbb{C}^n)\rightarrow Flags(\mathbb{C}^n)$ pulls back each tautological line bundle $E_i/E_{i-1}$ to itself.

Recall from \cite{k-t} that the $T^n$-equivariant Schubert class $[X_{\mu}(F_{\bullet})]\in H^{\ast}_{T^n}(Gr_s(\mathbb{C}^n))$ with respect to the standard reference flag $F_{\bullet}$ is represented by the factorial Schur function (see \cite{m-s}) in the $T^n$-equivariant cohomology of $Flags(\mathbb{C}^n)$ :
\begin{equation*}
p^*[X_{\mu}(F_{\bullet})]
= s_{\mu}(-\bar x_1,\cdots,-\bar x_s | -t_n,\cdots,-t_1) .
\end{equation*}
For the convenience of the reader, we here recall the definition of factorial Schur functions from \cite{m-s}: for a Young diagram $\mu$ with at most $s$ rows, the factorial Schur function associated to $\mu$ is defined to be 
\begin{equation*}
s_{\mu}(x_1,\cdots,x_s | a_1,a_2, \cdots)
= \sum_{T} \prod_{\alpha\in\mu} (x_{T(\alpha)}-a_{T(\alpha)+c(\alpha)})
\end{equation*}
as a polynomial in $\mathbb{Z}[x_1,\cdots,x_s]\otimes\mathbb{Z}[a_1,a_2,\cdots]$
where $T$ runs over all semistandard tableaux of shape $\mu$ with entries in $\{1,\cdots,s\}$, $T(\alpha)$ is the entry of $T$ in the cell $\alpha\in\mu$, and $c(\alpha)=j-i$ is the content of $\alpha=(i,j)$.
This polynomial is symmetric in $x$-variables.

From the definition, we have that $X_{\mu}(\wN\tilde{F}_{\bullet})=\wN w_0 X_{\mu}(F_{\bullet})$ where $w_0\in S_n$ is the longest element with respect to the Bruhat order.
So it follows that 
\begin{align*}
p^*\tilde{S}_{\mu} 
&= p^*((\wN w_0)^{-1})^*[X_{\mu}(F_{\bullet})]
= ((\wN w_0)^{-1})^*p^*[X_{\mu}(F_{\bullet})] \\
&= s_{\mu}(-\bar x_1,\cdots,-\bar x_s | -t_{\wN(1)},\cdots,-t_{\wN(n)}) 
\end{align*}
since the projection $p : Flags(\mathbb{C}^n) \rightarrow Gr_s(\mathbb{C}^n)$ is equivariant with respect to the left $S_n$-actions.
In particular, the following lemma with the definition (\ref{eq:4.1}) shows that 
\begin{align}\label{eq:5.3}
p^*\tilde{S}_{\mu_{s,\dpara}} 
= (-1)^{\dpara} e_{\dpara}(\bar x_1,\cdots,\bar x_s | t_{\wN(1)},\cdots,t_{\wN(n)}).
\end{align}
\vspace{5pt}
\begin{lem}\label{fact Schur by elem sym}
For indeterminates $x_1,\cdots,x_s, a_1,a_2,\cdots$, we have
\begin{align*}
s_{\mu_{s,k}}(x_1,\cdots,x_s | a_1,a_2,\cdots) 
= \sum_{r=0}^k (-1)^{k-r} e_r(x_1,\cdots,x_s) h_{k-r}(a_1,\cdots,a_{s+1-k})
\end{align*}
for $k\geq 0$ where $\mu_{s,k}=(1,\cdots,1,0,\cdots,0)$ with $1$ repeated $k$-times and $0$ repeated $(s-k)$-times.
\end{lem}
\proof
We first find the coefficient of the monomial $x_1\cdots x_r$ in $s_{\mu_{s,k}}(x|a)$.
For each 
$I=(i_1,i_2,\cdots,i_{k-r})$ satisfying $r+1\leq i_1<i_2<\cdots<i_{k-r}\leq s$, there is a summand in $s_{\mu_{s,k}}(x|a)$ corresponding to the standard tableau $T_{I}$ of shape $\mu_{s,k}$ whose $(j,1)$-th entry is
\begin{align*}
\begin{cases}
j \quad &\text{if $1\leq j\leq r$,} \\
i_{j-r} &\text{if $r+1\leq j\leq k$.}
\end{cases}
\end{align*}
The summand is of the form
\begin{align*}
(x_1-a_1)(x_2-a_1)\cdots(x_r-a_1)  (x_{i_1}-a_{i_1-r})(x_{i_2}-a_{i_2-r-1})\cdots(x_{i_{k-r}}-a_{i_{k-r}-k+1}),
\end{align*}
and the contribution of the monomial $x_1\cdots x_r$ from this polynomial is \begin{align*}
(-1)^{k-r}(a_{i_1-r}a_{i_2-r-1}\cdots a_{i_{k-r}-k+1})x_1\cdots x_r.
\end{align*}
Since the condition on $I$ is equivalent to  
\begin{align*}
1\leq i_1-r\leq i_2-r-1\leq \cdots\leq i_{k-r}-k+1\leq s-k+1,
\end{align*}
we see that the coefficient of $x_1\cdots x_r$ in $s_{\mu_{s,k}}(x_1,\cdots,x_s | a_1,a_2,\cdots)$ is 
\begin{align*}
(-1)^{k-r} h_{k-r}(a_1,\cdots,a_{s-k+1}).
\end{align*}
Recalling that $s_{\mu_{s,k}}(x_1,\cdots,x_s | a_1,a_2,\cdots)$ is symmetric in $x$-variables, we conclude that the coefficient of $x_{j_1}\cdots x_{j_{r}}$ is $(-1)^{k-r} h_{k-r}(a_1,\cdots,a_{s-k+1})$ for any $1\leq j_1<\cdots<j_r\leq s$. Thus, the polynomial
\begin{align*}
(-1)^{k-r} e_r(x_1,\cdots,x_s)h_{k-r}(a_1,\cdots,a_{s-k+1})
\end{align*}
gives the summand in $s_{\mu_{s,k}}(x_1,\cdots,x_s | a_1,a_2,\cdots)$ whose degree in $x$-variables is $r$. 
\qed

From now on, we take a specific choice of $\wN$ as follows, and we study the image of the Schubert classes $p^*\tilde{S}_{\mu}$ under $\rho_{\lambda}$.
We choose $\wN$ so that its one-line notation is given by
\begin{align*}
\wN = J_1 \cdots J_{\ell}
\end{align*}
where each sector $J_{\apfir}$ is a sequence of subsectors 
\begin{align*}
J_{\apfir} = j_{\apfir}^{(1)} \cdots j_{\apfir}^{(\lambda_{\apfir}-\lambda_{\apfir+1})}
\end{align*}
consisted by sequences of the form
\begin{align*}
j_{\apfir}^{(m)}=(\lambda_1-\lambda_{\apfir})+m \ , \ (\lambda_1-\lambda_{\apfir})+\lambda_2+m \ , \ \dots\dots \ , \ (\lambda_1-\lambda_{\apfir})+\lambda_2+\cdots+\lambda_{\apfir}+m.
\end{align*}
Note that $j_{\apfir}^{(m)}$ is a sequence of length $r$, and 
 $J_{\apfir}$ is a sequence of length $r(\lambda_{\apfir}-\lambda_{\apfir+1})$.
We define $J_{\apfir}$ to be the empty sequence if $\lambda_r=\lambda_{r+1}$. Writing down $J_{\apfir}$ for some small $r$, the reader can see how the complete flag $\wN F_{\bullet}$ refines the flag $(\cdots\subset N_0^2\mathbb{C}^n \subset N_0\mathbb{C}^n \subset \mathbb{C}^n)$.
\begin{ex}
If $n=16$ and $\lambda=(7,5,2,2)$, then
\begin{align*}
\wN = 1 \ 2 \ 3 \ 8 \ 4 \ 9 \ 5 \ 10 \ 6 \ 11 \ 13 \ 15 \ 7 \ 12 \ 14 \ 16
\end{align*}
where $J_1=j_1^{(1)}j_1^{(2)}=1 \ 2$, $J_2=j_2^{(1)}j_2^{(2)}j_2^{(3)}= 3 \ 8 \ 4 \ 9 \ 5 \ 10$, $J_3$ is the empty sequence, and $J_4=j_4^{(1)}j_4^{(2)}= 6 \ 11 \ 13 \ 15 \ 7 \ 12 \ 14 \ 16$. The reader should check that $\wN F_{\bullet}$ refines the flag $(\cdots\subset N_0^2\mathbb{C}^n \subset N_0\mathbb{C}^n \subset \mathbb{C}^n)$.
\end{ex}
The map  $\phi:[n]\rightarrow[\ell]$ defined in \eqref{eq:3.4} takes each sequence $j_{\apfir}^{(m)}$ to the sequence $1,\cdots,{\apfir}$ since $k$-th number of  $j_{\apfir}^{(m)}$ satisfies 
\begin{align*}
\lambda_1+\cdots+\lambda_{k-1}+1
\leq
(\lambda_1-\lambda_{\apfir})+\lambda_2+\cdots+\lambda_{k}+m 
\leq
\lambda_1+\cdots+\lambda_{k}.
\end{align*}
This shows that $\phi\circ \wN$ coincides with the map $\phi_{\lambda}$ defined in \eqref{eq:4.4}.
Taking $\rho_{\lambda}$ to \eqref{eq:5.3}, we obtain
\begin{equation*}
\rho_{\lambda}\circ p^*(\tilde{S}_{\mu_{s,\dpara}}) 
= (-1)^{\dpara} e_{\dpara}(\bar y_1,\cdots,\bar y_s | u_{\phi_{\lambda}(1)},\cdots,u_{\phi_{\lambda}(n)}) 
\end{equation*}
in $H^{\ast}_{T^\ell}(\mathcal{S}_{\lambda})$.
Since $i^*(\tilde{S}_{\mu_{s,j}})=0$, the commutative diagram \eqref{eq:5.2} shows that the left-hand-side of this equality vanishes.

\vspace{10pt}
Now, the homomorphism \eqref{eq:4.5} induces a surjective ring homomorphism 
\begin{equation*}
\bar\psi : \mathbb{Z}[\yy_1,\cdots,\yy_n, \uu_1,\cdots \uu_\ell]/\widetilde{I}_{\lambda}
\longrightarrow
H_{T^{\ell}}^*(\mathcal{S}_{\lambda}).
\end{equation*}
In what follows, we prove that this is an isomorphism by thinking of both sides as $\mathbb{Z}[\uu_1,\cdots \uu_\ell]$-algebras. Namely, the ring on the left-hand-side admits the obvious multiplication by $\uu_1,\cdots,\uu_n$, and the ring on the right-hand-side has the canonical ring homomorphism $H^{\ast}(BT^{\ell})\rightarrow H^{\ast}_{T^{\ell}}(\Spr)$ with the identification $H^{\ast}(BT^{\ell})=\mathbb{Z}[\uu_1,\cdots \uu_\ell]$.

Recall that $\mathcal{S}_{\lambda}$ admits a cellular decomposition by even dimensional cells (\cite{spa}).
So the spectral sequence for the fiber bundle $ET^{\ell}\times_{T^{\ell}}\mathcal{S}_{\lambda}\rightarrow BT^{\ell}$ shows that $H_{T^{\ell}}^*(\mathcal{S}_{\lambda})$ is a free $\mathbb{Z}[\uu_1,\cdots,\uu_{\ell}]$-module and that its rank over $\mathbb{Z}[\uu_1,\cdots,\uu_{\ell}]$ coincides with the rank of the non-equivariant cohomology:
\begin{align*}
\text{rank}_{\mathbb{Z}[\uu_1,\cdots,\uu_{\ell}]}H_{T^{\ell}}^*(\mathcal{S}_{\lambda})
=
\text{rank}_{\mathbb{Z}} H^*(\mathcal{S}_{\lambda})
= \frac{n!}{\lambda_1!\lambda_2!\cdots\lambda_{\ell}!} 
=: \binom{n}{\lambda}.
\end{align*}
Hence, to prove that the map $\bar\psi$ is an isomorphism, it is sufficient to show that the module $\mathbb{Z}[\yy_1,\cdots,\yy_n, \uu_1,\cdots \uu_\ell]/\widetilde{I}_{\lambda}$ is generated by $\binom{n}{\lambda}$ elements as a $\mathbb{Z}[\uu_1,\cdots,\uu_{\ell}]$-module.
To do that, let us consider a graded ring\footnote{The argument in \cite{t} to give a presentation of the ring $H^*(\mathcal{S}_{\lambda};\mathbb{C})$ works also over $\mathbb{Z}$-coefficient, and in that sense this ring is the presentation given in \cite{t}.}
$\mathbb{Z}[y_1,\cdots,y_n]/I_{\lambda}$ where $I_{\lambda}$ is Tanisaki's ideal, namely this is generated by $e_{d}(\yy_{i_1},\cdots,\yy_{i_s})$ for $1\leq s\leq n$, $1\leq i_1<\cdots<i_s\leq n$, and $d \geq s+1-p_{\checklambda }(s)$. In \cite{t}, it is shown that this is a free $\mathbb{Z}$-module of rank $\binom{n}{\lambda}$. 
\begin{lem}\label{gene}
Let $\Phi_1(y),\cdots,\Phi_k(y)$ be homogeneous polynomials in $\mathbb{Z}[\yy_1,\cdots,\yy_n]$ which give an additive basis of $\mathbb{Z}[\yy_1,\cdots,\yy_n]/I_{\lambda}$ where $k=\binom{n}{\lambda}$.
If we think of $\Phi_1(y),\cdots,\Phi_k(y)$ as elements of $\mathbb{Z}[\yy_1,\cdots,\yy_n,u_1,\cdots,u_{\ell}]/\tilde{I}_{\lambda}$, then they generate $\mathbb{Z}[\yy_1,\cdots,\yy_n,u_1,\cdots,u_{\ell}]/\tilde{I}_{\lambda}$ as a $\mathbb{Z}[u_1,\cdots,u_{\ell}]$-module. 
\end{lem}
\begin{proof}
It suffices to show that any monomial $m$ of $\yy_1,\cdots,\yy_n$ in $\mathbb{Z}[\yy_1,\cdots,\yy_n,u_1,\cdots,u_{\ell}]/\tilde{I}_{\lambda}$ can be written as a $\mathbb{Z}[u_1,\cdots,u_{\ell}]$-linear combination of $\Phi_1(y),\cdots,\Phi_k(y)$.
We prove this by induction on the degree $d$ of $m$.
The base case $d=0$ is clear, i.e.  $\Phi_i(y)=1$ for some $i$. We assume that $d\geq 1$ and the claim holds for $d-1$. Let $\theta$ be a homomorphism from $\mathbb{Z}[\yy_1,\cdots,\yy_n,u_1,\cdots,u_{\ell}]/\tilde{I}_{\lambda}$ to $\mathbb{Z}[\yy_1,\cdots,\yy_n]/I_{\lambda}$ sending $y_i$ to $y_i$ and $u_i$ to $0$. This is well-defined since each generator $e_{\dpara}(\bar y_1,\cdots,\bar y_s | u_{\phi_{\lambda}(1)},\cdots,u_{\phi_{\lambda}(n)})$ of $\tilde{I}_{\lambda}$ is mapped to the corresponding generator $e_{d}(\yy_{i_1},\cdots,\yy_{i_s})$ of $I_{\lambda}$.
By the assumption, $\theta(m)$ can be written as a $\mathbb{Z}$-linear combination of $\Phi_1(y),\cdots,\Phi_k(y)$, that is, we have 
\begin{align*}
m-\sum_{i} a_i\Phi_i(y)\in \ker \theta
\end{align*}
for some $a_i\in\mathbb{Z}$.
Here, $\ker \theta$ is the ideal of $\mathbb{Z}[\yy_1,\cdots,\yy_n,u_1,\cdots,u_{\ell}]/\tilde{I}_{\lambda}$ generated by $u_1,\cdots,u_{\ell}$.
In fact, it follows that the image of ${I}_{\lambda}$ in $\mathbb{Z}[\yy_1,\cdots,\yy_n,u_1,\cdots,u_{\ell}]/\tilde{I}_{\lambda}$ is included in the ideal $(u_1,\cdots,u_{\ell})$ of $\mathbb{Z}[\yy_1,\cdots,\yy_n,u_1,\cdots,u_{\ell}]/\tilde{I}_{\lambda}$ 
from the following equation in $\mathbb{Z}[\yy_1,\cdots,\yy_n,u_1,\cdots,u_{\ell}]/\tilde{I}_{\lambda}$:
\begin{align*}
e_d(\yy_{i_1},\cdots,\yy_{i_s})
&= -\sum_{0\leq r<d} (-1)^{d-r} e_r(\yy_{i_1},\cdots,\yy_{i_s}) h_{d-r}(u_{\phi_{\lambda}(1)},\cdots,u_{\phi_{\lambda}(s+1-d)}).
\end{align*}
Therefore, the monomial $m$ can be written as 
\begin{align}\label{eq:5.4}
m = \sum_{i} a_i\Phi_i(y)+\sum_{j=1}^{\ell}f_j(y,u)u_j
\end{align}
for some polynomials $f_1(y,u),\cdots,f_{\ell}(y,u)$.
Since $m$ has degree $d$, we can replace the polynomials in the right-hand-side by their homogeneous components of degree $d$. Namely, we can assume that $\deg\Phi_i(y)=\deg f_j(y,u)+1=d$.
Now, the induction assumption shows that each $f_j(y,u)$ is written as a $\mathbb{Z}[u_1,\cdots,u_{\ell}]$-linear combination of $\Phi_1(y),\cdots,\Phi_k(y)$ since 
the degree of each monomial in $y$ contained in $f_j(y,u)$ is less than $d$.
Hence, the element $m$ is written by a $\mathbb{Z}[u_1,\cdots,u_{\ell}]$-linear combination of $\Phi_1(y),\cdots,\Phi_k(y)$ in $\mathbb{Z}[\yy_1,\cdots,\yy_n,u_1,\cdots,u_{\ell}]/\tilde{I}_{\lambda}$, as desired.

\end{proof}
From Lemma~\ref{gene}, the surjection $\bar\psi$ has to be an isomorphism as discussed above.

\end{document}